\documentclass[12pt,fleqn]{article}

\usepackage{amsmath,epic,amssymb}
\usepackage{amsthm,amsfonts,amscd,epsfig}
\usepackage{xcolor}

\newtheorem{theorem}{Theorem}[section]
\newtheorem{lemma}[theorem]{Lemma}
\newtheorem{remark}[theorem]{Remark}
\newtheorem{proposition}[theorem]{Proposition}
\newtheorem{corollary}[theorem]{Corollary}
\newtheorem{definition}[theorem]{Definition}
\newtheorem{example}[theorem]{Example}
\newcount\refno
\newcommand\be{\begin{equation}}
\newcommand\ee{\end{equation}}
\newcommand\bn{\begin{eqnarray}}
\newcommand\en{\end{eqnarray}}
\newcommand\bns{\begin{eqnarray*}}
\newcommand\ens{\end{eqnarray*}}
\newcommand\bd{\begin{definition}}
\newcommand\ed{\end{definition}}
\newcommand\br{\begin{remark}}
\newcommand\er{\end{remark}}
\newcommand\bt{\begin{theorem}}
\newcommand\et{\end{theorem}}
\newcommand\bp{\begin{proposition}}
\newcommand\ep{\end{proposition}}
\newcommand\bc{\begin{corollary}}
\newcommand\ec{\end{corollary}}
\newcommand\bl{\begin{lemma}}
\newcommand\el{\end{lemma}}
\newcommand\pf{\begin{proof}}
\newcommand\bC{{\mathbb C}}
\newcommand\bK{{\mathbb K}}
\newcommand\bR{{\mathbb R}}
\newcommand\bN{{\mathbb N}}

\newcommand\cR{{\cal R}}

\newcommand{\F}{\mbox{$\mathcal{F}$}}

\def\d{{\rm d}}

\makeindex

\allowdisplaybreaks
\begin{document}

\title{The Vertical Recursive Relation of Riordan Arrays and Their Matrix Representation}
\author{Tian-Xiao He\\
{\small Department of Mathematics}\\
{\small Illinois Wesleyan University}\\
{\small Bloomington, IL 61702-2900, USA}\\
}
\date{}

\maketitle

\begin{abstract}
\noindent 
A vertical recursive relation approach to Riordan arrays is induced, while the horizontal recursive relation is represented by $A$- and $Z$-sequences. 
This vertical recursive approach gives a way to represent the entries of a Riordan array $(g,f)$ in terms of a recursive linear combinations of the coefficients of $g$. A matrix representation of the vertical recursive relation is also given. The set of all those matrices forms a group, called the quasi-Riordan group.  The extensions of the horizontal recursive relation and the vertical recursive relation in terms of $c$- and $C$- Riordan arrays are defined with illustrations by using the rook triangle and the Laguerre triangle. Those extensions represent a way to study nonlinear recursive relations of the entries of some triangular matrices from linear recursive relations of the entries of Riordan arrays. In addition, the matrix representation of the vertical recursive relation of Riordan arrays provides transforms between lower order and high order finite Riordan arrays, where the $m$th order Riordan array is defined by $(g,f)_m=(d_{n,k})_{m\geq n,k\geq 0}$. Furthermore, the vertical relation approach to Riordan arrays provides a unified approach to construct identities.


\end{abstract}

\pagestyle{myheadings} 
\markboth{T. X. He}
{The Vertical Recursive Relation of Riordan Arrays}

\section{Introduction}

Riordan matrices are infinite, lower triangular matrices defined by the generating function of their columns. With the matrix multiplication, they form a group, called {\em the Riordan group} (see Shapiro, Getu, Woan and Woodson \cite{SGWW}).

More formally, let us consider the set of formal power series ring $\F = K[\![$$t$$]\!]$, where ${\bK}$ is the field of ${\bR}$ or ${\bC}$. The \emph{order} of $f(t)  \in \F$, $f(t) =\sum_{k=0}^\infty f_kt^k$ ($f_k\in {\bK}$), is the minimum number $r\in\bN_0$ such that $f_r \neq 0$, where ${\bN}_0={\bN}\cup \{0\}$ and ${\bN}$ is the positive integer set. We denote by $\F_r$ the set of formal power series of order $r$. Let $g(t) \in \F_0$ and $f(t) \in \F_1$; the pair $(g,\,f )$ defines the {\em (proper) Riordan matrix} $D=(d_{n,k})_{n,k\in \bN_0}=(g, f)$ having
  
\begin{equation}\label{Radef}
d_{n,k} = [t^n]g(t) f(t)^k
\end{equation}
or, in other words, having $g f^k$ as the generating function of the $k$th column of $(g,f)$.  The {\it first fundamental theorem of Riordan matrices } means the action of the proper Riordan matrices on the formal power series presented by  

\[
(g(t), f(t)) h(t)=g(t) (h\circ f)(t),
\]
which can be abbreviated as $(g,f)h=gh(f)$. Thus we immediately see that the usual row-by-column product of two Riordan matrices is also a Riordan matrix:
\begin{equation}\label{Proddef}
    (g_1,\,f_1 )  (g_2,\,f_2 ) = (g_1 g_2(f_1),\,f_2(f_1)).
\end{equation}
The Riordan matrix $I = (1,\,t)$ is the identity matrix because its entries are $d_{n,k} = [t^n]t^k=\delta_{n,k}$.

Let $(g\left( t\right),\,f(t)) $ be a Riordan matrix. Then its inverse is

\begin{equation}
(g\left( t\right),\,f(t))^{-1}=\left( \frac{1}{g(\overline{{f}}(t))},
\overline{f}(t),\right)  \label{Invdef}
\end{equation}%
where $\overline {f}(t)$ is the compositional inverse of $f(t)$, i.e., $(f\circ 
\overline{f})(t)=(\overline{f}\circ f)(t)=t$. In this way, the set $\mathcal{
R}$ of all proper Riordan matrices forms a group (see \cite{SGWW}) called the Riordan group,

Here is a list of six important subgroups of the Riordan group (see \cite{SGWW, Sha, Barry, SSBCHMW}).

\begin{itemize}
\item the {\it Appell subgroup} $\{ (g(z),\,z)\}$.
\item the {\it Lagrange (associated) subgroup} $\{(1,\,f(z))\}$.
\item the {\it $k$-Bell subgroup} $\{(g(z),\, z(g(z))^k)\}$, where $k$ is a fixed positive integer. 
\item the {\it hitting-time subgroup} $\{(zf'(z)/f(z),\, f(z))\}$.
\item the {\it derivative subgroup} $\{ (f'(z), \, f(z))\}$.
\item the {\it checkerboard subgroup} $\{ (g(z),\, f(z))\},$ where $g$ is an even function and $f$ is an odd function. 
\end{itemize}

The $1$-Bell subgroup is referred to as the Bell subgroup for short, and the Appell subgroup can be considered as the $0$-Bell subgroup if we allow $k=0$ to be included in the definition of the $k$-Bell subgroup.

Let $G$ be a group, and let $H$ and $N$ be two subgroups of $G$ with $N$ normal. Then the following statements are equivalent: 

\begin{itemize}
\item $NH=G$ and $N\cap H={e}$.
\item Every $g\in G$ can be written uniquely as $g=nh$, where $n\in N$ and $h\in H$.
\item Define $\psi: H\to G/N$ by $\psi(h)=Nh$, $h\in H$. Then $\psi$ is an isomorphism. 
\end{itemize}
If these conditions hold, we write $G=N\rtimes H$ and say that $G$ is expressed as an semidirect product of $N$ and $H$. Since for every 
$(g,f)\in \mathcal {R}$ 

\[
(g,f)=(g,t)(1,f),
\]
where $(g,t)\in \mathcal{A}$ and $(1,f)\in \mathcal {L}$ and $\mathcal{A}$ and $\mathcal {L}$ are Appell subgroup, the normal subgroup of $\mathcal{R}$, and the Lagrange subgroup, respectively. 

An infinite lower triangular matrix $[d_{n,k}]_{n,k\in{\bN_0}}$ is a Riordan matrix if and only if a unique sequence $A=(a_0\not= 0, a_1, a_2,\ldots)$ exists such that for every $n,k\in{\bN_0}$  
\be\label{eq:1.1}
d_{n+1,k+1} =a_0 d_{n,k}+a_1d_{n,k+1}+\cdots +a_{n-k}d_{n,n}. 
\ee 
This is equivalent to 
\be\label{eq:1.2}
f(t)=tA(f(t))\quad \text{or}\quad t=\bar f(t) A(t).
\ee
Here, $A(t)$ is the generating function of the $A$-sequence. The first formula of \eqref{eq:1.2} is also called the {\it second fundamental theorem of Riordan matrices}. Moreover, there exists a unique sequence $Z=(z_0, z_1,z_2,\ldots)$ such that every element in column $0$ can be expressed as the linear combination 
\be\label{eq:1.3}
d_{n+1,0}=z_0 d_{n,0}+z_1d_{n,1}+\cdots +z_n d_{n,n},
\ee
or equivalently,
\be\label{eq:1.4}
g(t)=\frac{1}{1-tZ(f(t))},
\ee
in which and thoroughly we always assume $g(0)=g_0=1$, a usual hypothesis for proper Riordan matrices with normalization. From \eqref{eq:1.4}, we may obtain 

\[
Z(t)=\frac{g(\bar f(t))-1}{\bar f(t)g(\bar f(t))}.
\]

$A$- and $Z$-sequence characterizations of Riordan matrices were introduced, developed, and/or studied in Merlini, Rogers, Sprugnoli, and Verri \cite{MRSV}, Roger \cite{Rog}, Sprugnoli and the author \cite{HS}, Cheon and Jin \cite{CJ}, Cheon, Luz\'on, Mor\'on, Prieto Martinez, and Song \cite{CLMPMS}, \cite{He15}, Jean-Louis and Nkwanta \cite{JLN}, Luz\'on, Mor\'on, and Prieto-Martinez \cite{LMPM17, LMPM}, etc. In \cite{HS} the expressions of the $A$- and $Z$-sequences of the product depend on the analogous sequences of the two given factors.  

The Catalan numbers, $C_{n}=\binom{2n}{n}/(n+1)$, are a sequence of integers
that occur in many counting situations. The books \textit{Enumerative
Combinatorics,} Volume 2 and the more recent \textit{Catalan Numbers} by
Richard Stanley \cite{Stanley} and \cite{Sta} give a wealth of information.
The generating function, denoted by $C(z)$, of the Catalan numbers can be
written as

\begin{equation*}
C(t)=\sum_{n=0}^{\infty }C_{n}t^{n}=\frac{1-\sqrt{\mathstrut 1-4t}}{2t},
\end{equation*}
which is equivalent to the functional equation $C(t)=1+tC(t)^{2}$. It can be
shown that \cite{GKP}

\begin{equation*}
C(t)^{k}=\sum_{n=0}^{\infty }\frac{k}{2n+k}\binom{2n+k}{n}t^{n}.  
\end{equation*}

A Dyck path of length $2n$ is a path in the plane lattice $\mathbb{Z}\times 
\mathbb{Z}$ from the origin $(0,0)$ to $(2n,0)$, made up with steps $(1,1)$
and $(1,-1).$ The other requirement is that a path can never go below the $x$%
-axis. We refer to $n$ as the semilength of the path. It is well known that
the number of Dyck paths of semilength $n$ is the $n$-th Catalan number $%
C_{n}$. A partial Dyck path is a Dyck path without requiring that end point
be on the $x$-axis. Hence, 

\be\label{1.8-3}
C(n,k):=[t^n]C(t)^k=\frac{k}{2n+k}\binom{2n+k}{n}
\ee
is the number of the partial Dyck paths from $(0,0)$ to $(2n+k,k)$.

For a positive integer $m$, an $m$-Dyck path is a path from the origin to $%
(mn,0)$ using the steps $(1,1)$ and $(1,1-m)$ and again not going below the $%
x$-axis. We refer to $mn$ as the length of the path. A partial $m$-Dyck path
is defined as an $m$-partial Dyck path. It is well known that the number of $%
m$-Dyck paths of length $mn$ is (see, for example, \cite{GKP})

\begin{equation*}
F_{m}(n,1)=\frac{1}{mn+1}\binom{mn+1}{n}, 
\end{equation*}
the Fuss-Catalan numbers. For $m=2$, the Fuss-Catalan numbers are the
Catalan numbers $F_{2}(n,1)$. More generally, the Fuss-Catalan
numbers are

\begin{equation}\label{1.8-4}
F_m(n,r):=\frac{r}{mn+r}\binom{mn+r}{n},
\end{equation}
which are named after N. I. Fuss and E. C. Catalan (see \cite%
{FL,GKP,LP,Mlo,PZ, He12, HS17}). The Fuss-Catalan numbers have many combinatorial
applications (see, for example, Shapiro and the author \cite{HS17}). 

The generating function $F_{m}(t)$ for the Fuss-Catalan numbers, $%
\{F_{m}(n,$ $1)\}_{n\geq 0}$ is called the generalized binomial series in \cite%
{GKP}, and it satisfies the function equation $F_{m}(t)=1+tF_{m}(t)^{m}$.
Hence from Lambert's formula for the Taylor expansion of the powers of $%
F_{m}(t)$ (see \cite{GKP}), we have that

\begin{equation}
F_{m}^{r}\equiv F_{m}(t)^{r}=\sum_{n\geq 0}\frac{r}{mn+r}\binom{mn+r}{n}t^{n}
\label{1.4}
\end{equation}%
for all $r\in {{\mathbb{R}}}$. The key case\eqref{1.4} leads the
following formula for $F_{m}(t)$:

\begin{equation}
F_{m}(t)=1+tF_{m}^{m}(t).  \label{1.5}
\end{equation}%
Actually,

\begin{eqnarray*}
1+tF_{m}^{m}(t) &=&1+\sum_{n\geq 0}\frac{m}{mn+m}\binom{mn+m}{n}t^{n+1} \\
&=&1+\sum_{n\geq 1}\frac{m}{mn}\binom{mn}{n-1}t^{n} \\
&=&\sum_{n\geq 0}\frac{1}{mn+1}\binom{mn+1}{n}t^{n}=F_{m}(t).
\end{eqnarray*}%
For the cases $m=1$ and $2$, we have $F_{1}=1/(1-t)$ and $F_{2}=C(z)$,
respectively. When $m=3$, the Fuss-Catalan numbers $\left( F_{3}\right) _{n}$
form the sequence $A001764$ (see \cite{Sloane}), $1,1,3,12,55,273,1428,%
\ldots $, the ternary numbers. The ternary numbers count the number of $3$%
-Dyck paths or ternary paths. The generating function of the ternary numbers
is denoted as $T(t)=\sum_{n=0}^{\infty }T_{n}t^{n}$ with $T_{n}=\frac{1}{3n+1%
}\binom{3n+1}{n}$, and is given equivalently by the equation

\begin{equation*}
T(t)=1+tT(t)^{3}.  
\end{equation*}

\section{A vertical recursive relation view of Riordan arrays}

In a Riordan array $(g,f)=(d_{n,k})_{n,k\geq 0}$, the first column ($0$th column), with its generating function $g(t)$, usually possesses an interesting combinatorial interpretation or represents an important sequence, while the other columns might be considered as the compositions of the first column in the view of the following recurrence relation:

\be\label{2.1}
d_{n,k}=[t^n]gf^k=\sum^n_{j= 1} f_j[t^{n-j}]gf^{k-1}=\sum^{n-k+1}_{j=1}f_jd_{n-j,k-1}
\ee
for $n,k\geq 1$, where $f(t)=\sum_{j\geq 1} f_j t^j$. For instance, 

\begin{align*}
&d_{1,1}=f_1 d_{0,0},\\
&d_{2.1}=f_2 d_{0,0}+f_1d_{1,0},\,\, d_{2,2}=f_1d_{1,1},\\
&d_{3,1}=f_3d_{0,0}+f_2d_{1,0}+f_1d_{2.0},\,\, d_{3,2}=f_2d_{2,1}+f_1d_{1,1}, \,\, d_{3,3}=f_1d_{2,2},\ldots
\end{align*}

\begin{example}\label{ex:2.1}
For the Pascal triangle $(1/(1-t),t/(1-t))$, from \eqref{2.1} we obtain the well-known recursive relation 

\be\label{2.1.2}
\binom{n}{k}=\sum^{n-k+1}_{j=1}\binom{n-j}{k-1}.
\ee
\end{example}

Recurrence relation \eqref{2.1} provides a resource for constructing identities and an algorithm for computing powers and multiplications of formal power series. The latter approach may simplifies the corresponding computation process by using Fa\`a di Bruno's formula. 

\begin{theorem}\label{thm:2.1}
Let $(g,f)=(d_{n,k})_{n,k\geq 0}$ be a Riordan array, and let $C(n,k)$ and $F_m(n,r)$ be defined by \eqref{1.8-3} and  \eqref{1.8-4}, respectively. 
Then we obtain the recurrence relation \eqref{2.1}. Particularly, for $g=F_m$ we have 

\begin{align}\label{2.2}
&\frac{k+1}{m(n-k)+k+1}\binom{m(n-k)+k+1}{n-k}\nonumber\\
=&\sum^{n-k}_{j=0}\frac{k}{(mj+1)(m(n-j-k)+k)}\binom{mj+1}{j}\binom{m(n-j-k)+k}{n-k-j}.
\end{align}
If $m=2$, then $g=F_2=C$ and \eqref{2.2} becomes

\begin{align}\label{2.3}
&\frac{k+1}{2n-k+1}\binom{2n-k+1}{n-k}\nonumber\\
=&\sum^{n-k}_{j=0}\frac{k}{(2j+1)(2(n-j-k)+k)}\binom{2j+1}{j}\binom{2(n-j-k)+k}{n-k-j},
\end{align}
or equivalently, 

\be\label{2.4}
C(n-k,k+1)=\sum^{n-k}_{j=0}C(j,1)C(n-j-k,k),
\ee
where $C(m,\ell)$ are defined by \eqref{1.8-3}, and $C(j,1)=[t^j]C(t)$, the $j$th Catalan number. 
\end{theorem}

\begin{theorem}\label{thm:2.2}
Let $(g,f)=(d_{n,k})_{n,k\geq 0}$ be a Riordan array. Then its entries $d_{n,k}$, $n,k\geq 0$, can be represented recursively by 
\begin{align}\label{2.5}
d_{n,k}=&\sum^{n-(k-1)}_{i_k=1}f_{i_k}\sum^{n-(k-2)-i_k}_{i_{k-1}=1}f_{i_{k-1}}\sum^{n-(k-3)-i_k-i_{k-1}}_{i_{k-2}=1}f_{i_{k-2}}\cdots \nonumber\\
&\sum^{n-i_k-i_{k-1}-i_2}_{i_1=1}f_{i_1}g_{n-i_1-i_2-\cdots -i_k}.
\end{align}
\end{theorem}

Let $A$ and $B$ be $m\times m$ and $n\times n$ matrices, respectively. The we define the direct sum of $A$ and $B$ by 

\be\label{2.5+2}
A\oplus B=\left[ \begin{matrix} A & 0\\ 0 & B\end{matrix}\right]_{(m+n)\times (m+n)}.
\ee

\begin{definition}\label{def:2.3}
Let $g\in \F_0$ with $g(0)=1$ and $f\in \F_1$. We call the following matrix a quasi-Riordan array and denote it by $[g,f]$.

\be\label{2.6}
[g,f]:=(g,f,tf,t^2f,\ldots),
\ee
where $g$, $f$, $tf$, $t^2f\cdots$ are the generating functions of the $0$th, $1$st, $2$nd, $3$rd, $\cdots$, columns of the matrix $[g,f]$, respectively. It is clear that $[g,f]$ can be written as 

\be\label{2.8}
[g,f]=\left( \begin{matrix} g(0) & 0\\ (g-g(0))/t & (f,t)\end{matrix}\right),
\ee
where $(f,t)=(f,tf,t^2f,t^3f,\ldots)$. Particularly, if $f=tg$, then the quasi-Riordan array $[g,tg]=(g, t)$, a Appell-type Riordan array. 
\end{definition}

Note that $[g,f]$ defined by \eqref{2.6} is not the Riordan-Krylov matrix defined in \cite{CS}, which relationship is worth being 
investigated. 

\begin{theorem}\label{thm:2.4}
Let $(g,f)=(d_{n,k})_{n,k\geq 0}$ be a Riordan array, and let $([1]\oplus (g,f))$ and $[g,f]$ be defined by \eqref{2.5} and \eqref{2.6}, respectively. Then $(g,f)$ has the expression

\be\label{2.9}
(g,f)=[g,f]([1]\oplus (g,f)).
\ee
\end{theorem}

\begin{proof} We write the formal power series $g$ and $f$ in the Riordan array $(g,f)$ as $g=\sum_{n\geq 0} g_n t^n$ and $f=\sum_{n\geq 1} f_n t^n$. Then
\begin{align}\label{2.10}
&\left( \begin{matrix} g(0) & 0\\ (g-g(0))/t & (f,t)\end{matrix}\right)([1]\oplus (g,f))\nonumber\\
=&\left( \begin{matrix} d_{0,0} & & & &  \\d_{1,0} &f_1 & & &  \\d_{2,0} &f_2 & f_1 & & \\ d_{3,0} & f_3 & f_2 & f_1 & \\
d_{4,0}& f_4& f_3 & f_2& f_1\\\vdots & \vdots& \vdots& \vdots &\ddots \end{matrix}\right) \left( \begin{matrix} 1 & & & &  \\0 &d_{0,0} & & &  \\0 &d_{1,0} & d_{1,1} & & \\ 0 & d_{2,0} & d_{2,1} & d_{2,2} & \\0& d_{3,0}&d_{3,1}& d_{3,2} & d_{3,3}\\\vdots & \vdots& \vdots& \vdots &\ddots\end{matrix}\right) \nonumber\\
=& \left( \begin{array} {lllll}d_{0,0} & & & & \\d_{1,0} & f_1d_{0,0} & & &\\  d_{2,0} & f_2d_{0,0}+f_1d_{1,0} & f_1d_{1,1} & &\\d_{3,0}&f_3d_{0,0}+f_2d_{1,0}+f_1d_{2,0}& f_2d_{1,1}+f_1d_{2,1}& f_1d_{2,2}&\\ 
\vdots & \vdots& \vdots& \vdots &\ddots\end{array}\right)\nonumber\\
=&\left( \begin{matrix} d_{0,0} & & & & \\d_{1,0} & d_{1,1} & & &\\  d_{2,0} & d_{2,1} & d_{2,2} & &\\d_{3,0}&d_{3,1}& d_{3,2} & d_{3,3}& \\d_{4,0} &d_{4,1}& d_{4,2}& d_{4,3} & d_{4,4}\\\vdots & \vdots& \vdots& \vdots &\ddots\end{matrix}\right),
\end{align}
where the last step follows \eqref{2.1}, which completes the proof of \eqref{2.6}.
\end{proof}

For an integer $r\geq 0$ denote the $r$-th truncations of a power series $h=\sum_{n\geq 0} h_n t^n$ by $h|_r:=\sum^r_{n=0} h_nt^n$.

\begin{corollary}\label{cor:2.5}
Let $(g,f)=(d_{n,k})_{n,k\geq 0}$ be a Riordan array with $g=\sum_{n\geq 0} g_n t^n$ and $f=\sum_{n\geq 1} f_n t^n$, and let $(g,f)_n$ be the $n$th order leading principle submatrix of $(g,f)$. Then we have the recursive relation of $(g,f)_n$ in the following form:

\be\label{2.11}
(g,f)_n=[g,f]_n([1]\oplus (g,f)_{n-1}),
\ee
where $[g,f]_n$ is the $n$th order leading principle submatrix of the quasi-Riordan array $[g,f]$ defined by \eqref{2.5}, namely 

\be\label{2.12}
[g,f]_n=\left( \begin{matrix} g(0) & 0\\ ((g-g(0))/t)|_{n-1} & (f,t)_{n-1}\end{matrix}\right),
\ee 
where the $n-1$st truncation of $(g-g(0)/t)|_{n-1}$ is $(g-g(0)/t)|_{n-1}=\sum^{n-1}_{k=1}g_kt^{k-1}$, and $(f,t)_{n-1}=(f, tf, t^2f, \ldots, t^{n-1} f).$ We call $[g,f]_n$ the recursive matrix of the Riordan array $(g,f)$. 
\end{corollary}

Mao, Mu, and Wang \cite{MMW} uses \eqref{2.12} gives another interesting criterion for the total positivity of Riordan arrays. 

In the next section, we will show that the set of all quasi-Riordan arrays forms a group, called the quasi-Riordan group. 

\section{The quasi-Riordan group}

Denote by ${\cal R}_r$ the set of all quasi-Riordan arrays defined by \eqref{2.6}. 
In this section, we will show ${\cal R}_r$ is a group with respect to regular matrix multiplication. 

\begin{proposition}\label{pro:4.1}
Let $[g,f]$, $[d,h]\in {\cal R}_r$, and let $u=\sum_{n\geq 0} u_n t^n\in \F_0$. Then there holds the first fundamental theorem for quasi-Riordan arrays (FFTQRA)

\be\label{4.1}
[g,f]u=gu(0)+\frac{f}{t}(u-u(0)),
\ee
which implies 

\be\label{4.2}
[g,f][d,h]=\left[g+\frac{f}{t}(d-1), \frac{fh}{t}\right].
\ee

Hence, $[1,t]$ is the identity of ${\cal R}_r$. 
\end{proposition}

\begin{proof} The FFTQRA \eqref{4.1} can be proved as follows. 
\begin{align*}
&[g,f]u=(g, f, tf, t^2f, \cdots )\left( \begin{matrix}u_0\\u_1\\u_2\\ \vdots\end{matrix}\right)\\
=&gu_0 +f\sum_{n\geq 1} u_nt^{n-1}=gu_0+\frac{f}{t}(u-u_0).
\end{align*}

By using FFTQRA and noting $d(0)=1$ and $h(0)=0$, we have 

\begin{align*}
& [g,f][d,h]=(g, f, tf, t^2f, \cdots)(d, h, th, t^2h,\cdots)\\
=&\left( g+\frac{f}{t}(d-1), \frac{f}{t}h, \frac{f}{t}th, \frac{f}{t}t^2h,\cdots\right),
\end{align*}
which implies \eqref{4.2}. 

Substituting $[g,f]=[1,t]$ and $[d,h]=[1,t]$ into \eqref{4.2}, we obtain, respectively, 

\begin{align*}
&[1,t][g,f]=[1+(g-1), f]=[g,f]\quad \mbox{and}\\
&[g,f][1,t]=\left[ g+\frac{f}{t}(1-1), \frac{f}{t}t\right]=[g,f],
\end{align*}
which implies $[1,t]$, the identity matrix, is the identity of ${\cal R}_r$. 
\end{proof}

\begin{theorem}\label{thm:4.2}
The set of all quasi-Riordan arrays ${\cal R}_r$ is a group, called the quasi-Riordan group, with respect to the multiplication represented in \eqref{4.2}.
\end{theorem}

\begin{proof}
From \eqref{4.2} of Proposition \ref{pro:4.1}, ${\cal R}_r$ is closed with respect to the multiplication. \eqref{4.2} also 
shows $[1,t]$ is the identity of ${\cal R}_r$. For any $[g,f]\in {\cal R}_r$, its inverse is 

\be\label{4.3}
[g,f]^{-1}=\left[ 1+\frac{t}{f}(1-g), \frac{t^2}{f}\right],
\ee
since

\bn
\left[ 1+\frac{t}{f}(1-g), \frac{t^2}{f}\right][g,f]=\left[ 1+\frac{t}{f}(1-g)+\frac{t^2/f}{t}(g-1), \frac{t^2}{f}\frac{f}{t}\right]=[1,t].
\en
Finally, the associative law is satisfied for any $[g,f], [d,h]$, and $[u,v]\in{\cal R}_r$ because 

\begin{align*}
&([g,f][d,h])[u,v]=\left[ g+\frac{f}{t}(d-1), \frac{fh}{t}\right][u,v]\\
=&\left[ g+\frac{f}{t}(d-1)+\frac{fh}{t^2}(u-1), \frac{fhv}{t^2}\right]
\end{align*}
and

\begin{align*}
&[g,f]([d,h])[u,v]=[g,f] \left[ d+\frac{h}{t}(u-1), \frac{hv}{t}\right]\\
=&\left[ g+\frac{f}{t}\left(d+\frac{h}{t}(u-1)-1\right), \frac{fhv}{t^2}\right]\\
=&\left[ g+\frac{f}{t}(d-1)+\frac{fh}{t^2}(u-1), \frac{fhv}{t^2}\right]
\end{align*}
shows $([g,f][d,h])[u,v]=[g,f]([d,h][u,v])$.
\end{proof}

\begin{example}\label{ex:4.1}
Consider the quasi-Riordan array $[1/(1-t), t/(1-t)]$, which is a Appell-type Riordan array $(1/(1-t),t)$ with its inverse $(1-t,t)$. From \eqref{4.3}  its inverse in ${\cal R}_r$ is 

\begin{align*}
&\left[ \frac{1}{1-t},\frac{t}{1-t}\right]^{-1}=\left[ 1+\frac{t}{f}(1-g), \frac{t^2}{f}\right]\\
=&\left[ 1+\frac{t}{t/(1-t)}\left(1-\frac{1}{1-t}\right), \frac{t^2}{t/(1-t)}\right]=[1-t, t(1-t)]\\
=&(1-t,t)=\left( \begin{matrix} 1 & & &   \\-1 &1 & &   \\0 &-1 & 1 &  \\ 0 & 0 & -1 & 1  \\
\vdots & \vdots& \vdots &\ddots \end{matrix}\right).
\end{align*}
\end{example}

\begin{theorem}\label{thm:4.3}
A quasi-Riordan array $[g,f]$ is a Riordan array if and only if $f=tg$, i.e., when $[g,f]$ is an Appell-type Riordan array. Hence, 
${\cal A}_r:=\{ [g,tg]: g\in \F_0, g(0)=1\}$ is a subgroup of ${\cal R}_r$, which is called the Apell quasi-Riordan subgroup.
\end{theorem}

\begin{proof}
Let $[g,tg]$, $[d,td]\in {\cal R}_r$. Then 

\[
[g,tg][d,td]=\left[ g+\frac{tg}{t}(d-1), \frac{t^2gd}{t}\right]=[ gd,tgd].
\]
${\cal A}$ is closed under the multiplication. In addition, the inverse of $[g,f]$

\[
[g,tg]^{-1}= \left[ 1+\frac{t}{tg}(1-g),\frac{t^2}{tg}\right]=\left[ \frac{1}{g}, \frac{t}{g}\right]
\]
is also in ${\cal A}$.
\end{proof}

\begin{theorem}\label{thm:4.4}
Let ${\cal R}_r$ be the quasi-Riordan group. Then every conjugate of an element $[g,f]\in {\cal R}_r$ is in the set ${{\cal R}(f)}_{r}:=\{[d,f]: d\in \F_0, d(0)=1\}$. Hence, ${\cal A}=\{ [g,t]:g\in \F_0, g(0)=1\}$ is a normal subgroup of ${\cal R}_r$. 
\end{theorem}

\begin{proof}
Let $[d,h]\in {\cal R}_r$. Then for an arbitrarily fixed $[g,f]\in{\cal R}_r$, we have 

\begin{align*}
&[d,h][g,f][d,h]^{-1}=[d,h][g,f]\left[ 1+\frac{t}{h}(1-d), \frac{t^2}{h}\right]\\
=&\left[ d+\frac{h}{t}(g-1)-\frac{f}{t}(d-1), f\right]\in {{\cal R}(f)}_r.
\end{align*}
Particularly, if $f=t$ in $[g,f]$, then 

\begin{align*}
&[d,h][g,t][d,h]^{-1}=\left[ 1+\frac{h}{t}(g-1), t\right],
\end{align*}
which implies ${\cal A}=\{ [g,t]:g\in \F_0, g(0)=1\}$ is a normal subgroup of ${\cal R}_r$.
\end{proof}

\section{An vertical recursive relation view of $(c)$- and $(C)$-Riordan arrays}
In this section, we will extend the vertical recursive relation to the $(c)$ Riordan arrays. We start from the following definition of $(c)$-class of Riordan arrays. 

\begin{definition}\label{def:3.1}
Let $(g,f)=(d_{n,k})_{n,k\geq 0}$ be a Riordan array, and let $c = (c_0, c_1, c_2, . . .)$ be a sequence satisfying $c_0 = 1$ and $c_k \not= 0$ for all $k =1,2,\ldots$. Denote 

\be\label{3.1-0}
(g,f)_c=\left(\frac{c_n}{c_k}d_{n,k}\right)_{n,k\geq 0}.
\ee
Then the set 
\be\label{3.1}
\{{\cR}\}_{c}:=\left\{(g,f)_c: (g,f)=(d_{n,k})_{n,k\geq 0}\in{\cR}\right\} 
\ee
is called the $(c)$-class of ${\cR}$ or the set of the $(c)$-Riordan arrays. Since we may change $c_n$ and $c_k$ respectively to $1/c_n$ and $1/c_k$, 

\[
\left(\frac{c_k}{c_n}d_{n,k}\right)_{n,k\geq 0}=\left(\frac{1/c_n}{1/c_k}d_{n,k}\right)_{n,k\geq 0}
\]
is in the $(1/c)$-class of ${\cR}$, where $(1/c)=(1/c_0, 1/c_1, 1/c_2,\ldots)$.

Let $(g,f)=(d_{n,k})_{n,k\geq 0}$ be a Riordan array, and let $C = (c_{n,k})_{n,k\geq 0}$ be a lower triangle matrix satisfying $c_{n,0}= 1$ and $c_{n,k}\not= 0$ for all $0\leq k\leq n$ and $c_{n,k}=0$ for all $k> n$. Denote 

\be\label{3.2-0}
(g,f)_C=\left( \frac{c_{n,n}}{c_{n,k}}d_{n,k}\right)_{n,k\geq 0}.
\ee
The set 

\be\label{3.2}
\{{\cR}\}_C:=\left\{(g,f)_C: (g,f)=(d_{n,k})_{n,k\geq 0}\in{\cR}\right\}
\ee
is called the $(C)$-class of ${\cR}$, or the set of the $(C)$-Riordan arrays. Similarly, $\left(\frac{c_{n,k}}{c_{n,n}}d_{n,k}\right)_{n,k\geq 0}$ is in the $(1/C)$-class of ${\cR}$, where $(1/C)=(1/c_{n,k})_{n,k\geq 0}$. 
\end{definition}

Wang and Wang \cite{WW} defines the $(c)$-Riordan arrays by using the generalized formal power series. Gould and the author \cite{GH} claims that for a sequence $(c)$ the $(c)$-class of ${\cR}$, or the set of the $(c)$-Riordan arrays $\{{\cR}\}_c$, forms a group. More precisely, we have the following theorem. 

\begin{theorem}\label{thm:3.-1}\cite{GH}
Let $c = (c_0, c_1, c_2, . . .)$ be a sequence satisfying $c_0 = 1$ and $c_k \not= 0$ for all $k =1,2,\ldots$, and let $\{{\cR}\}_c$ be the $(c)$-class defined by \eqref{3.1}. Then $\{{\cR}\}_c$ 
forms a group with respect to the regular matrix multiplication. This group is also denoted by ${\cR}_c$ and called the $(c)$-Riordan group with respect to the sequence $(c)$.
\end{theorem}

\begin{proof}
If $(g_1,f_1)_c, (g_2,f_2)_c\in \{(g,f)\}_c$, and denote $(g_1,f_1)=(d_{n,k})_{n,k\geq 0}$, $(g_2,f_2)=(e_{n,k})_{n,k\geq 0}$, and $(g_1,f_1)(g_2,f_2)=(g_1g_2(f_1),f_2(f_1))=(h_{n,k})_{n,k\geq 0}$, then $(d_{n,k})_{n,k\geq 0}(e_{n,k})_{n,k\geq 0}=(h_{n,k})_{n,k\geq 0}$, and 

\begin{align*}
(g_1,f_1)_c(g_2,f_2)_c=&\left(\sum^{\min{n,k}}_{j=0}\left( \frac{c_n}{c_j}d_{n,j}\right)\left( \frac{c_j}{c_k}e_{j,k}\right)\right)_{n,k\geq 0}\\
=& \left( \frac{c_n}{c_k}\sum^{\min{n,k}}_{j=0}(d_{n,j}e_{j,k})\right)_{n,k}\\
=& \left(\frac{c_n}{c_k}h_{n,k}\right)_{n,k\geq 0}=(g_1g_2(f_1), f_2(f_1))_c.
\end{align*}
Hence, we may find $(1.t)_c$ is the identity of ${\cR}_c$ and the inverse of an element $(g,f)_c\in {\cR}_c$ is $(1/g(\overline{f}),\overline{f})$, where $\overline{f}$ is the compositional inverse of $f$. 
Finally, we have 

\begin{align*}
\left( (g,f)_c(d,h)_c\right)(u,v)_c=&\left( gd(f)u(h(v)), v(h(f))\right)\\
=&(g,f)_c\left( (d,h)_c(u,v)_c\right).
\end{align*}
completing the proof.
\end{proof}

More materials on $(c)$-Riordan arrays can be found in \cite{He22, SSBCHMW}.

\begin{example}\label{ex:3.1}
Let $(g,f)=(1/(1-t), t/(1-t))$. Then $(1/(1-kt), t/(1-kt))\in \{(1/(1-t), t/(1-t))\}_c$ with $c=(c_n)_{n\geq 0}=(k^n)_{n\geq 0}$. The $(c)$-Riordan array $(1/(1-kt), t/(1-kt))_c$ begins 

\[
\left[\begin{array}{cccccccc}
1 & 0 & 0 & 0 & 0 & 0 & 0 &\\ 
k & 1 & 0 & 0 & 0 & 0 & 0 &\\ 
k^2 & k^2 & 1 & 0 & 0 & 0 &  0&\\ 
k^3 & k^2+k^3 & k^2+k& 1 & 0 & 0 & 0&\\
k^4 &2k^3+k^4 & 2k^2+2k^3 & k^2+2k& 1 & 0 & 0& \\ 
k^5 & 3 k^4+k^5& 4k^3+3k^4 & 4k^2+3k^3& k^2+3k & 1 &0&  \\ 
\vdots & \vdots& \vdots &\vdots &\vdots   & \vdots & \ddots
\end{array}
\right] 
\]
\end{example}

\begin{example}\label{ex:3.2}
Let $(g,f)=(1/(1-t), t/(1-t))=(\binom{n}{k})_{n,k\geq 0}$, and let $c=(1/n!)_{n\geq 0}$. Then 

\be\label{ex:3.2-2}
(r_{n,k})_{n,k\geq 0}=\left(\frac{n!}{k!}\binom{n}{k}\right)_{n,k\geq 0}=\left( (n-k)!\left( \binom{n}{k}\right)^2\right)\in \{(1/(1-t), t/(1-t))\}_c,
\ee
where $c=(c_n)_{n\geq 0}=(n!)_{n\geq 0}$. The $(c)$-Riordan array $(r_{n,k})_{n,k\geq 0}$ begins 

\[
(r_{n,k})_{n,k\geq 0}=\left[\begin{array}{cccccccc}
1 & 0 & 0 & 0 & 0 & 0 & 0 &\\ 
1 & 1 & 0 & 0 & 0 & 0 & 0 &\\ 
2 & 4 & 1 & 0 & 0 & 0 &  0&\\ 
6 & 18 & 9& 1 & 0 & 0 & 0&\\
24&96 & 72 & 16& 1 & 0 & 0& \\ 
120 & 600& 600 & 200& 5 & 1 &0&  \\ 
\vdots & \vdots& \vdots &\vdots &\vdots   & \vdots & \ddots
\end{array}
\right].
\]
The matrix $(r_{n,k})_{n,k\geq 0}$ is called a rook matrix. The polynomial $r_n(x)=\sum^n_{k=0}r(n,k)x^{n-k}$ are called the rook polynomial of $n$th order (cf. Fielder \cite{Fie} and Riordan \cite{Rio}). 

In general, consider a board that represents a full  or a part chess board. Let $m$ be the number of squares present in the board. Two pawns or rooks placed on a board are said to be in non capturing positions if they are not in same row or same column. For $ 2 \leq k\leq m$, let $r_k$ denote the number of ways in which $k$ rooks can be placed on a board such that no two rooks capture each other. Then the polynomial $1+r_1x+r_2x^2+\cdots +r_mx^m$ is called the (general) rook polynomial for the board considered. If the board is denoted by $C$, then the corresponding polynomial is denoted by $r(C,x)$. The rook polynomials are defined for $m\geq 2$. If $m=1$, then the board contains only one square so that $r_k=0$ for $k\geq 2$, and $r(C,x)=1+x$. Hence, the coefficient $r_{n,k}$ of the roof polynomial $r_n(x)=\sum^n_{k=0}r(n,k)x^{n-k}$ is referred to as the number of ways in which $k$ rooks can be placed on a $n\times n$ full chess board such that no two rooks capture each other. The rook polynomials for $1\times 1$, $2\times 2$, $3\times 3$, $4\times 4$, and $5\times 5$ full boards are, respectively, 

\begin{align*}
&r_1(x)=x+1,\\
&r_2(x)=2x^2+4x^1+1,\\
&r_3(x)=6x^3+18x^2+9x+1,\\
&r_4(x)=24x^4+96x^3+72x^2+16x+1,\\
&r_5(x)=120x^5+600x^4+600x^3+200x^2+5x+1.
\end{align*}

In a given board C, suppose we choose a particular square denoted as $(\ast)$. Let $D$ denote the board obtained from $C$ by deleting the row and column containing the square $(\ast)$, and let $E$ be the board obtained from $C$ by deleting only the square $(\ast)$. Then the rook polynomial 
for the board $C$ is given by $r(C,x)=xr(D,x)+r(E,x)$. This is known as Expansion formula for $r(C,x)$. 

Denote the $(n+1)\times (n+1)$ full board by $C_{n+1}$. Let $n\times n$ full board $C_n$ be located at the left upper corner of $C_{n+1}$, and let $E_n$ be the board obtained from the board $C_{n+1}$ by deleting only the square at the right lower corner of $C_{n+1}$. Then, we have 

\begin{align}\label{3.2-2}
r(E_n,x)=&r(C_{n+1}, x)-xr(C_n,x)=r_{n+1}(x)-xr_n(x)\nonumber\\
=&\sum^{n+1}_{k=0}\frac{(n+1)!}{k!}\binom{n+1}{k}x^{n+1-k}-x\sum^n_{k=0}\frac{n!}{k!}\binom{n}{k}x^{n-k}\nonumber\\
=&\sum^{n+1}_{k=0}\frac{n!}{k!}\left((n+1)\binom{n+1}{k}-\binom{n}{k}\right) x^{n+1-k}\nonumber\\
=&\sum^{n}_{k=0}\frac{n!}{k!}\binom{n}{k}\left(\frac{(n+1)^2}{n-k+1}-1\right) x^{n+1-k}+1\nonumber\\
=&\sum^{n}_{k=0}\frac{n!}{k!}\frac{n^2+n+k}{n-k+1}\binom{n}{k}x^{n+1-k}+1\nonumber\\
=&\sum^{n}_{k=0}\frac{n^2+n+k}{n-k+1}(n-k)!\left(\binom{n}{k}\right)^2x^{n+1-k}+1.
\end{align}
We denote $r(E_n,x)=\sum^{n+1}_{k=0}E_{n,k}x^{n+1-k}$ and call it the remainder polynomial, where 

\be\label{3.2-3}
E_{n,k}=\frac{n^2+n+k}{n-k+1}(n-k)!\left(\binom{n}{k}\right)^2
\ee
for $0\leq k\leq n$ and $E_{n,n+1}=1$. Then, the lower triangle matrix $(E_{n,k})_{n,k\geq 0}$, called the remainder triangle, begins

\[
(E_{n,k})_{n,k\geq 0}=\left[\begin{array}{cccccccc}
0& 1 & 0 & 0 & 0 & 0 & 0 &\\ 
1 & 3 & 1 & 0 & 0 & 0 & 0 &\\ 
4& 14 & 8 & 1& 0 & 0 &  0&\\ 
18 & 78 & 63& 15 & 1 & 0 & 0&\\
96&504 &528 & 184& 4 & 1 & 0& \\ 
\vdots & \vdots& \vdots &\vdots &\vdots   & \vdots & \ddots
\end{array}
\right].
\]
\end{example}

Here is a recursive relation related to the rook triangle $(r_{n,k})_{n,k\geq 0}$ and the triangle 
$(E_{n,k})_{n,k\geq 0}$. 

\begin{proposition}\label{pro:3.0}
Let $(r_{n,k})_{n,k\geq 0}$ and $(E_{n,k})_{n,k\geq 0}$ be the rook triangles and the remainder triangle defined in Example \ref{ex:3.2}, and let $r_n(x)=\sum^n_{k=0}r(n,k)x^{n-k}$ and $r(E_n,x)=\sum^n_{k=0}E_{n,k}x^{n-k}$ be the rook polynomial and the remainder polynomial, respectively. Then, there exists recursive relation 

\be\label{3.2-4}
r_{n+1}(x)=xr_n(x)+r(E_n,x)
\ee
for $n\geq 0$, which can be expressed as a matrix form

\be\label{3.2-5}
\overline{(r_{n,k})}=(r_{n,k})+(E_{n,k}),
\ee
where $n\geq 0$ and $\overline{(r_{n,k})}$ is obtained from $(r_{n,k})$ by deleting its first row. 
Furthermore, we have the expansion formula for $r_{n+1}(x)$ as

\be\label{3.2-6}
r_{n+1}(x)=\sum^n_{k=0} x^{n-k}r(E_k,x)+x^{n+1}.
\ee
\end{proposition}

\begin{proof}
Equation \eqref{3.2-4} comes from the Expansion formula. Substituting $r_n(x)=\sum^n_{k=0}r(n,k)x^{n-k}$ and $r(E_n,x)=\sum^n_{k=0}E_{n,k}x^{n-k}$ into \eqref{3.2-4} and comparing the coefficients of the same powers of $x$ on the both sides yields 

\[
r_{n+1,k}=r_{n,k}+E_{n,k},
\]
which can be combined as the matrix form \eqref{3.2-5} for $n,k\geq 0$. From \eqref{3.2-4} we have 

\begin{align*}
&r_{n+1}(x)=xr_n(x)+r(E_n,x),\\
&xr_{n}(x)=x^2r_{n-1}(x)+xr(E_{n-1},x),\\
&x^2r_{n-1}(x)=x^3r_{n-2}(x)+x^2r(E_{n-2},x),\\
&\cdots,\\
&x^{n}r_1(x)=x^{n+1}r_0(x)+x^{n}r(E_0,x).
\end{align*}
Adding up the above system and cancelling the same terms from the both sides yields \eqref{3.2-6}.
\end{proof}

\begin{example}\label{ex:3.3}
Let $(g,f)=(1/(1-t), t/(1-t))=(\binom{n}{k})_{n,k\geq 0}$. Then 

\begin{align}\label{ex:3.3-2}
(L_{n,k})_{n,k\geq 0}=&\left(\frac{(-1)^n/(n)_n}{(-1)^k /(n)_k}\binom{n}{k}\right)_{n,k\geq 0}\nonumber\\
=&\left( \frac{(-1)^{n-k}}{(n-k)!}\binom{n}{k}\right)\in \{(1/(1-t), t/(1-t))\}_C,
\end{align}
where $C=(c_{n,k})_{n,k\geq 0}=((-1)^n/(n)_k)_{n,k\geq 0}$. The $(C)$-Riordan array $(L_{n,k})_{n,k\geq 0}$ begins 

\[
\left[\begin{array}{cccccccc}
1 & 0 & 0 & 0 & 0 & 0 & 0 &\\ 
-1 & 1 & 0 & 0 & 0 & 0 & 0 &\\ 
\frac{1}{2}& -2 & 1 & 0 & 0 & 0 &  0&\\ 
-\frac{1}{6} & \frac{3}{2} & -3& 1 & 0 & 0 & 0&\\
\frac{1}{24}&-\frac{2}{3} &3 & -4& 1 & 0 & 0& \\ 
-\frac{1}{120} &\frac{5}{24}&-\frac{5}{3} & 5& -5 & 1 &0&  \\ 
\vdots & \vdots& \vdots &\vdots &\vdots   & \vdots & \ddots
\end{array}
\right].
\]
The matrix $(L_{n,k})_{n,k\geq 0}$ is called the Laguerre matrix. The polynomial $L_n(n,k)=\sum^n_{k=0} L_{n,k}x^{n-k}$ are called the Laguerre polynomial of order $n$ (cf., for example, Hsu, Shiue, and the author \cite{HHS} and Riordan \cite{Rio}).

Comparing \eqref{ex:3.2-2} and \eqref{ex:3.3-2} and noting 

\[
r_{n,k}=\frac{n!}{k!}\binom{n}{k},
\]
we immediately obtain

\be\label{ex:3.3-3}
r_{n,n-k}=\frac{n!}{(n-k)!}\binom{n}{k}=\frac{n!}{(n-k)!} (-1)^{n-k} (n-k)! L_{n,k}=(-1)^{n-k}n!L_{n,k},
\ee
which implies 

\be\label{ex:3.3-4}
\sum^n_{k=0}r_{n,n-k}x^{n-k}=n!\sum^n_{k=0} L_{n,k}(-x)^{n-k}=n!L_n(-x)
\ee
for $0\leq k\leq n$. Since the left-hand side of \eqref{ex:3.3-4} can be written as 

\begin{align*}
\sum^n_{k=0}r_{n,n-k}x^{n-k}=&\sum^n_{k=0}r_{n,k}x^k=x^n\sum^n_{k=0}r_{n,k}\left( \frac{1}{x}\right)^{n-k}=x^nr_n\left( \frac{1}{x}\right).
\end{align*}
Hence, we have the well-known formula 

\[
r_n(x)=n!x^nL_n\left( -\frac{1}{x}\right)
\]
for $n\geq 0$.
\end{example}

\begin{definition}\label{def:3.0}
We may extend the rook triangle and the rook polynomials to a general case by defending  the following generalized rook triangle and generalized rook polynomials: 

\be\label{ex:3.2-2-2}
(\hat r_{n,k})_{n,k\geq 0}=\left(\frac{n!}{k!}d_{n,k}\right)_{n,k\geq 0}=\left( (n-k)! \binom{n}{k}d_{n,k}\right)\in \{(g(t), f(t)\}_c,
\ee
where $c=(c_n)_{n\geq 0}=(n!)_{n\geq 0}$ and $(g,f)=(d_{n,k})_{n,k\geq 0}$. In addition, the polynomial 

\be\label{ex:3.2-2-3}
\hat r_n(x)=\sum^n_{k=0}\hat r_{n,k}x^{n-k}
\ee
is referred to as the generalized rook polynomial of degree $n$. 

We may also extend the Laguerre triangle and the Laguerre polynomials to a general case by defending  the following generalized Riordan type Laguerre triangle and generalized Riordan type Laguerre polynomials: 

\begin{align}\label{ex:3.3-2-2}
(\hat L_{n,k})_{n,k\geq 0}=&\left(\frac{(-1)^n/(n)_n}{(-1)^k/(n)_k}d_{n,k}\right)_{n,k\geq 0}\nonumber\\
=&\left( \frac{(-1)^{n-k}}{(n-k)!} \binom{n}{k}d_{n,k}\right)\in \{(g(t), f(t)\}_C,
\end{align}
where $C=(c_{n,k})_{n,k\geq 0}=(n!)_{n\geq 0}$ and $(g,f)=(d_{n,k})_{n,k\geq 0}$. In addition, the polynomial 

\be\label{ex:3.3-2-3}
\hat L_n(x)=\sum^n_{k=0}\hat L_{n,k}x^{n-k}
\ee
is referred to as the generalized Laguerre polynomial of degree $n$. 
\end{definition}

Similar to the relationship between the rook triangle and the Laguerre triangle and the relationship between the rook polynomials and the Laguerre polynomials, we may extend the relationships to the general cases as follows. 

\begin{theorem}\label{thm:3.0}
Let the generalized rook triangle and the rook polynomials and the generalized Riordan type Laguerre triangle and the generalized Riordan type Laguerre polynomials be defined in Definition \ref{def:3.0}. Then we have 

\be\label{ex:3.3-4-0}
\hat r_{n,n-k}=(-1)^{n-k}n!\hat L_{n,k}
\ee
for $0\leq k\leq n$ and 

\[
\hat r_n(x)=n!x^n\hat L_n\left( -\frac{1}{x}\right)
\]
for $n\geq 0$.
\end{theorem}

From Examples \ref{ex:3.2} and \ref{3.3}, we have a relationship between 

\begin{theorem}\label{thm:3.2}
Let $(g,f)=(d_{n,k})_{n,k\geq 0}$ be a Riordan array, and let $(g,f)_c$ and $(g,f)_C$ be defined in Definition \ref{def:3.1}. Then $(g,f)_c=(d^{(c)}_{n,k})_{n,k\geq 0}$ and $(g,f)_C=d^{(C)}_{n,k\geq 0}$ satisfy the horizontal recursive relations, which are extensions 
of $A$- and $Z$-sequence representation \eqref{eq:1.1} and \eqref{eq:1.3} of $(g,f)$ to $(g,f)_c$ and $(g,f)_C$, respectively:

\be\label{3.3}
d^{(c)}_{n+1,k+1}=\frac{c_{n+1}}{c_nc_{k+1}}\sum^{n-k}_{j=0}a_jc_{k+j}d^{(c)}_{n,k+j}
\ee
for $k\geq 0$ and 

\be\label{3.4}
d^{(c)}_{n+1,0}=\frac{c_{n+1}}{c_n}\sum^{n}_{j=0}z_jc_jd^{(c)}_{n,j}
\ee
as well as 

\be\label{3.5}
d^{(C)}_{n+1,k+1}=\frac{c_{n+1,n+1}}{c_{n,n}c_{n+1,k+1}}\sum^{n-k}_{j=0}a_jc_{n, k+j}d^{(C)}_{n,k+j}
\ee
for $k\geq 0$ and 

\be\label{3.6}
d^{(C)}_{n+1,0}=\frac{c_{n+1,n+1}}{c_{n,n}}\sum^{n}_{j=0}z_jc_{n,j}d^{(C)}_{n,j}.
\ee
\end{theorem}

\begin{theorem}\label{thm:3.3}
Let $(g,f)=(d_{n,k})_{n,k\geq 0}$ be a Riordan array, and let $(g,f)_c$ and $(g,f)_C$ be defined in Definition \ref{def:3.1}. Then $(g,f)_c=(d^{(c)}_{n,k})_{n,k\geq 0}$ and $(g,f)_C=d^{(C)}_{n,k\geq 0}$ satisfy the vertical recursive relations, which are extensions 
of vertical recursive relation \eqref{2.1} of the entries of $(g,f)$ to $(g,f)_c$ and $(g,f)_C$, namely, 

\be\label{3.7}
d^{(c)}_{n,k}=\frac{c_{n}}{c_{k}}\sum^{n-k+1}_{j=1}f_j\frac{c_{k-1}}{c_{n-j}}d^{(c)}_{n-j,k-1}
\ee
and 

\be\label{3.8}
d^{(C)}_{n,k}=\frac{c_{n,n}}{c_{n,k}}\sum^{n-k+1}_{j=1}f_j\frac{c_{n-j,k-1}}{c_{n-j,n-j}}d^{(C)}_{n-j,k-1},
\ee
respectively. 
\end{theorem}

\begin{example}\label{ex:3.4} It is known that the $A$- and $Z$- sequences of $(1/(1-t), t/(1-t))$ are $(1,1,0,\ldots)$ and $(1,0,0,\ldots)$, respectively. Then for the rook triangle represented in Example \ref{ex:3.2} with $d_{n,k}=\binom{n}{k}$, from \eqref{3.3} and \eqref{3.4}, we obtain the horizontal recursive relations for the entries of the rook triangle,

\be\label{3.9}
r_{n,k}=nr_{n-1,k}+\frac{n}{k}r_{n-1,k-1}
\ee
for $k\geq 1$ and 

\be\label{3.10}
r_{n,0}=nr_{n-1,0}.
\ee
From \eqref{3.7} and noting \eqref{2.1.2} we obtain the vertical recursive relation for the entries of the rook triangle 

\be\label{3.11}
r_{n,k}=\sum^{n-k+1}_{j=1}\frac{(n)_j}{k}r_{n-j,k-1}
\ee
for $k\geq 1$. 

From \eqref{ex:3.2-2}, one may obtain an equivalent identity of \eqref{3.9}

\[
\binom{n}{k}^2=\frac{n}{n-k}\binom{n-1}{k}^2+\frac{n}{k}\binom{n-1}{k-1}^2.
\]
Simialrly, from \eqref{3.11}

\[
\binom{n}{k}^2=\sum^{n-k+1}_{j=1}\frac{(n)_j}{k(n-k)_{j-1}}\binom{n-j}{k-1}^2.
\]
\end{example}

\begin{example}\label{ex:3.5} 
Similarly, for the Laguerre triangle represented in Example \ref{ex:3.3} with $d_{n,k}=\binom{n}{k}$, from \eqref{3.3} and \eqref{3.4}, we obtain the horizontal recursive relations for the entries of the Laguerre triangle,

\be\label{3.9-2}
L_{n,k}=L_{n-1,k-1}-\frac{1}{n-k}L_{n-1,k}
\ee
for $k\geq 1$ and 

\be\label{3.10-2}
L_{n,0}=-\frac{1}{n}L_{n-1,0}.
\ee
From \eqref{3.7} and noting \eqref{2.1.2} we obtain the vertical recursive relation for the entries of the Laguerre triangle 

\be\label{3.11-2}
L_{n,k}=\frac{1}{(n-k)!}\sum^{n-k+1}_{j=1}(-1)^{j-1}(n-k-j+1)!L_{n-j,k-1}
\ee
for $k\geq 1$.
It can be seen that \eqref{3.11-2} is equivalent to \eqref{2.1.2}. 
\end{example}

\end{document}